\documentclass{article}
\usepackage[utf8]{inputenc}
\usepackage{ amssymb }
\usepackage{ amsmath }
\usepackage{amsthm}
\usepackage{cleveref}
\usepackage{fancyhdr}
\fancypagestyle{footer}{\fancyhf{}\fancyfoot[L]{
L. J. Jolliffe, Department of Pure Mathematics and Mathematical Statistics, Centre for Mathematical Sciences, University of Cambridge, Wilberforce Road, Cambridge, CB3 0WB, United Kingdom\\
\emph{E-mail address:}  ljj33@cam.ac.uk}}
\newtheorem{theorem}{Theorem}

\newtheorem{definition}[theorem]{Definition}
\newtheorem{prop}[theorem]{Proposition}
\newtheorem{lemma}[theorem]{Lemma}
\newtheorem{cor}[theorem]{Corollary}

\crefalias{prop}{proposition}
\theoremstyle{plain}
\usepackage{titling}
\usepackage{xcolor}
\usepackage{tikz}
\usepackage{cleveref}
\newtheorem*{remark}{Remark}
\newcommand{\Sn}[0]{\mathcal{S}_n}
\newcommand{\val}[0]{\textup{val}_p}

\newcommand{\modp}[0]{\hspace{2pt} ( \textit{mod } p )}
\title{A combinatorial approach to first degree cohomology of Specht modules}
\author{Liam Jolliffe}
\date{}
\begin{document}
\maketitle
\begin{abstract}
 Using purely combinatorial methods we calculate the first degree cohomology of Specht modules indexed by two part partitions over fields of characteristic $p\ge 3$. These combinatorial methods also allow us to obtain an explicit description of all of the non-split extensions of the Specht module, $S^\lambda$, by the trivial module. Applying this work to partitions with more than two parts we are able to give an entirely combinatorial proof of the bound on the dimension of the first degree cohomology given by work of Donkin and Geranios. We also obtain as a corollary a result of Weber giving a far reaching condition determining partitions for which the first cohomology of the Specht module is trivial.   
\end{abstract}

\section{Introduction}
We shall briefly review some concepts from the representation theory of the symmetric group in this section, but we refer the reader to James' book \cite{James}, from which our notation is taken, for more detail. Let $\lambda\vdash n$ be a partition and let $S^\lambda$ be the corresponding Specht module for the symmetric group $\Sn$. The Specht module is a submodule of the transitive permutation module $M^\lambda$. The cohomology $H^i(\Sn,S^\lambda)=\text{Ext}^i_{\Sn} (k,S^\lambda)$ is known for $i\le 1$. Indeed, $H^0(\Sn,S^\lambda)=\text{Hom}_{k\Sn} (k,S^\lambda)$ and is determined module $M^\lambda$, as $\text{Hom}_{k\Sn} (k,S^\lambda)$ is contained in the one dimensional $\text{Hom}_{k\Sn} (k,M^\lambda)$. This means that calculating $H^0(\Sn,S^\lambda)$ is equivalent to determining if the trivial submodule of $M^\lambda$ is also contained in $S^\lambda$, which is an entirely combinatorial task, via James' Kernel Intersection Theorem [\Cref{Kernel Intersection Theorem}]. 

The first cohomology, $H^1(\Sn,S^\lambda)$, is understood due to Donkin and Geranios \cite{DG}, whose method is to compare these with the cohomology for the general linear group $\mathcal{GL}_n$, which reduces the calculation to that of $\text{Ext}^1_{B_n}(S^dE, K_\lambda)$, where $B_n$ is a Borel subgroup of $\mathcal{GL}_n$, $S^dE$ denotes the $d$th symmetric power of the natural module $E$ for $\mathcal{GL}_n$, and $K\lambda$ denotes the one dimensional $B_n$-module with weight $\lambda$. 

Hemmer \cite{Hemmer} had suggested an alternative approach for calculating $H^1(\Sn,S^\lambda)$ similar to the approach of James \cite{James 77} in calculating $H^0(\Sn,S^\lambda)$, which is the one we shall take in this paper. This approach is based on the observation that, like $H^0$, the first cohomology $H^1$ is also determined by $M^\lambda$ when the field $k$ has odd characteristic, as in this case any non-split extension of $S^\lambda$ by the trivial module, $k$, embeds into $M^\lambda$. This approach is entirely combinatorial and remains within the setting of the representation theory of the symmetric group. The other benefit of this approach is that it also gives an explicit description of the non-split extensions of $S^\lambda$ by $k$, which is new. We will complete the calculation of $H^1(\Sn,S^\lambda)$ in the case that $\lambda$ is a two part partition, recovering the result of Donkin and Geranios in this case \cite{DG}. We remark that Hemmer had also calculated $H^1(\Sn,S^\lambda)$ for two part partitions  \cite{Hemmer2}, although not via the combinatorial approach he suggested which motivated this paper; instead he shows that the calculation follows from work of Erdmann \cite{Erdmann} on the cohomology of $SL_2$, and so the proof does not remain in the setting of the symmetric group.

Recall that given a partition $\lambda\vdash n$, a $\lambda$-tableau is a bijection from $[n]:=\lbrace 1,\dots,n\rbrace $ to $[\lambda]$, the Young diagram of shape $\lambda$. A $\lambda$-tabloid is an equivalence class of $\lambda$-tableaux under the relation of row equivalence: $t\sim_R s$ if the entries in each row of $t$ are the same as the entries in the corresponding row of $s$. These equivalence classes will be denoted by writing the name of the tableau in braces, $\{t\}$. There is an obvious action of the symmetric group $\Sn$ on the set of $\lambda$-tabloids, by permuting the entries of a tabloid. We extend this set to a vector space over a field $k$ by taking formal sums of $\lambda$-tabloids and we call the resulting space the permutation module $M^\lambda$. In fact, we may extend this and define $M^\lambda$ in the case where $\lambda$ is a \emph{composition} of $n$, $\lambda\vDash n$, rather than a partition; that is $\sum_{i=0}^r \lambda_i =n$, but we do not require that the $\lambda_i$ are non-increasing.

Let $\lambda=(\lambda_1,\lambda_2,\dots,\lambda_r)\vdash n$ and let $i,v\in \mathbb{N}$ be such that $i<r$ and $v\le \lambda_i$. Let $\lambda'=(\lambda_1,\dots,\lambda_{i-1},\lambda_i+\lambda_{i+1}-v,v,\lambda_{i+2},\dots,\lambda_r)$. Define the homomorphism $\psi_{i,v}:M^\lambda\to M^{\lambda'}$ by
$$\psi_{i,v}(\{t\})=\sum_{t' \in T_{i,v}} \{t'\},$$
where the sum is over all those $\{t'\}$ who agree with $\{t\}$ on all rows other than rows $i$ and $i+1$, and whose $(i+1)$th row is a subset of the $(i+1)$th row of $\{t\}$. 
The Specht module $S^\lambda$ can be described as the span of set of distinguished elements of $M^\lambda$, the polytabloids- signed sums of tabloids obtained by the action on a tableaux by its `column symmetriser'. Equivalently the Specht module $S^\lambda$ can be characterised as follows \cite{James 77}:
\begin{theorem}[Kernel Intersection Theorem]\label{Kernel Intersection Theorem}
$$S^\lambda=\bigcap_{i=1}^{r-1}\bigcap_{v=0}^{\lambda_i-1}\text{Ker}(\psi_{i,v})\subseteq M^\lambda.$$
\end{theorem} 

Specht modules can be viewed as part of a large family of submodules of $M^\lambda$. If $\lambda$ is a composition of $n$ and $\mu$ a partition with $\mu_{i+1}\le \mu_i\le \lambda_i$ we may define an object ``between" a polytabloid and a tabloid by restricting our column symmetriser to the part of the tableaux which has shape $\mu$. The module spanned by these objects is denoted $S^{\mu,\lambda}$, and $S^\lambda=S^{\lambda,\lambda}$, while $M^\lambda=S^{(0),\lambda}$. These modules can be characterised as the common kernel of homomorphisms, similar to the characterisation of Specht modules via James' kernel intersection theorem. These modules, have a filtration of by Specht modules, which is described in \cite[Chapter 17]{James} and is used in the final section of this paper. 

Denote the sum of all $\lambda$-tabloids by $f_\lambda$ and observe that $f_\lambda\in M^\lambda$ is fixed by the action of $\Sn$. Clearly $H^0(\Sn,S^\lambda)$ is one dimensional if $f_\lambda\in S^\lambda$ and is 0 otherwise, which allows us to prove the following: 
\begin{theorem}\label{H0}
$H^0(\Sn,S^\lambda)$ is one dimensional if ${\lambda_i+j \choose j}\equiv 0 \modp$ for $1 \le j \le \lambda_{i+1}$ for all $i<r$ and is 0 otherwise.
\end{theorem}
We call a partition satisfying the condition above a \emph{James partition}. We will provide equivalent characterisations of James partitions in \Cref{binomial}.
Hemmer proved a similar result for the first cohomology \cite{Hemmer}:
\begin{theorem}\label{Hemmer}
Let $p\ge 3$ and $\lambda=(\lambda_1,\dots,\lambda_r)\vdash n$, then $\text{Ext}^1(k,S^\lambda)\ne 0$ if and only if there is an element $u\in M^\lambda$ with the following properties:
\begin{enumerate}
\item For each $1\le i < r$ and $0\le v < \lambda_i$, $\psi_{i,v}(u)\in M^{\lambda'}$ is a multiple of $f_{\lambda'}$, at least one of which is a non-zero multiple.
\item There does not exist a scalar $c\in k$ such that all the $\psi_{i,v}(c\cdot f_\lambda - u)=0$.
\end{enumerate}
If such a $u$ exists then the subspace $\langle S^\lambda,u\rangle \subseteq M^\lambda$ spanned by $S^\lambda$ and $u$ is a non-split extension of the Specht module $S^\lambda$ by the trivial module.
\end{theorem}
We will call an element $u\in M^\lambda$ satisfying the above conditions \emph{Hemmer}. Over fields of characteristic $p\ge 3$, any non-split extension of $S^\lambda$ by $k$ is contained in $M^\lambda$, and so for any non split extension we have such a $u$. The second condition ensures that $\langle S^\lambda,u\rangle $ is not the direct sum of $S^\lambda$ and a trivial module, and is automatic when $H^0(\Sn,S^\lambda)\ne 0$. In \cite{Weber}, Weber uses this method to give a far reaching combinatorial condition which sufficient for first degree cohomology to be trivial. In section 2 shall use Hemmer's method to calculate the first cohomology of in the case where $\lambda$ is a two part partition, $\lambda=(a,b)$, making extensive use of the theory of universal $p$-ary designs developed by the author in \cite{Jolliffe}.

In section 3 we construct a Hemmer element in $M^{(a,b)}$ in the special case that $(a,b)$ is a pointed, generalising the example given by Nguyen \cite{Ha Thu}, which was in turn a generalisation of the example in Hemmer's original paper \cite{Hemmer}. We also construct a Hemmer element when $\lambda=(a,b)$ is James, which is the only other case when Hemmer elements can exist. This gives a complete classification of all non-split extensions of Specht modules indexed by two part partitions.  

We conclude by considering the implications of these results for partitions of more than two parts, as a Hemmer element for such a partition must be Hemmer on restriction to some pair of adjacent rows. This gives a number of corollaries on the dimension of $H^1(\mathcal{S}_{n},S^\lambda)$ and in particular gives the main result of \cite{Weber}, which gives a class of partitions for which $\text{dim}(H^1(\mathcal{S}_{n},S^\lambda))=0$. Further to this we obtain an upper bound on $\text{dim}(H^1(\mathcal{S}_{n},S^\lambda))$ in general, consistent with the calculations of Donkin and Geranios \cite{DG}. 

\section{Hemmer Elements and Designs}

There is a deep connection between the representation theory of the symmetric group and combinatorial designs, which were studied by the author in \cite{Jolliffe}. 
\begin{definition}
Let $[v]:=\{1,2,\dots,v\}$ be a finite set and $1\le s \le v$ be an integer. A  \emph{universal $p$-ary design for $(v-s,s)$} is a function $$u:[v]_s\to k,$$ where $[v]_s$ is the set of all subsets of $[v]$ of size $s$, such that for all $i<s$
$$\hat{u}(Z):=\sum_{Y\supseteq Z}u(Y) = \mu_i \hspace{2em}\forall Z \in [v]_i $$
We call the $\mu_i$'s the \emph{coefficients} of the design and if $\mu_i=0$ for all $i<s$ then we say $u$ is a \emph{null-design}. We say that $\hat{u}$ is \emph{induced} from $u$.
\end{definition}
If the designs $u$ and $w$ have coefficients $\mu_i$ and $\nu_i$ respectively, then we say $u$ and $w$ are \emph{similar} if there is some $c\in k$ such that $c \cdot \mu_i=\nu_i$ for all $i<s$. 

Observe that if $\lambda=(a,b)$ is a two part partition then any element $u\in M^\lambda$ may be considered as a function $u':[a+b]_b\to k$ by first identifying each tabloid $\{t\}$ appearing in $u$ with the set of elements appearing in its second row, $t_2$, and then setting $u'(t_2)$ to be the coefficient of $\{t\}$ in $u$. The homomorphisms $\psi_{1,i}:M^{(a,b)}\to M^{(a+b-i,i)}$ correspond to the map taking $u'$ to its induced function $\hat{u'}$ restricted to sets of size $i$. From this new point of view, James' Kernel intersection theorem then determines all null universal $p$-ary designs for $(a,b)$.

In the setting of two part partitions, the first condition of \Cref{Hemmer} is precisely the requirement that $u'$ is a non-null universal $p$-ary design, while the second condition is to say that $u'$ is not similar to the constant design, which takes the value 1 on all sets of size $b$. Thus, finding a Hemmer element is equivalent to finding non-null universal $p$-ary designs which are not similar to the constant design. The author, in \cite{Jolliffe}, has completely characterised all the universal $p$-ary designs which exist for $(a,b)$. The existence, or otherwise of such designs depends on the shape of the partition $(a,b)$.

Let $a=\sum_{i=0}^\alpha a_ip^i$ be the base $p$ expansion of $a$; that is $0\le a_i\le p-1$ and $a_\alpha\ne 0$. The $p$-adic \emph{valuation} $\val(a)$ is the least $i$ such that $a_i$ is non-zero, we call $\alpha$ the $p$-adic \emph{length} of $a$ and write $l_p(a)=\alpha$. 
\begin{definition}
Let $(a,b)$ be a two part partition, that is $a\ge b >0$. We call a partition James if $\val(a+1)>l_p(b)$, while if $b=p^\beta + \hat{b}$ and $\hat{b} <p^{\val(a+1)} < p^\beta$ we call $(a,b)$ pointed.
\end{definition}
 The following lemma gives another characterisation of James partitions, in particular, $\lambda=(\lambda_1,\lambda_2,\dots,\lambda_r)$ is James if and only if $\lambda_i\equiv -1 \hspace{2pt} ( \emph{mod } p^{l_p(\lambda_{i+1})} )$ for all $i<r$, or equivalently $l_p(\lambda_i) < \val(\lambda_{i+1}+1)$ for all $i<r$.
\begin{lemma}\label{binomial}\cite{James}
Let $a,b\in \mathbb{N}$. The binomial coefficients ${a+1 \choose 1},{a+2 \choose 2},\dots{a+b \choose b}$ are all divisible by $p$ if and only if $a\equiv -1 \hspace{2pt} ( \emph{mod } p^{l_p(b)} )$.
\end{lemma}

\begin{theorem}\cite{Jolliffe} \label{combinatorial main theorem}
Let $a,b\in \mathbb{N}$, with $a\ge b$ and let $u$ be a non-null universal $p$-ary design for $(a,b)$. If $(a,b)$ is neither pointed or James, then $u$ is similar to the constant design. If $(a,b)$ is James then $u$ is unique up to similarity, while if $(a,p^\beta+\hat{b})$ is pointed then $u=u'+c$ where $u'$ is non-null only as a $\hat{b}$-design, while $c$ is similar to the constant design. 
\end{theorem}

\begin{cor}\label{main result}
Let $\lambda=(a,b)\vdash n$, and $p\ge 3$ then
$$
\text{dim}(H^1(\mathcal{S}_{n},S^\lambda))=
\begin{cases}
1 & \text{if } \lambda \text{ is James or pointed,} \\
0 & \text{otherwise}. 
\end{cases}
$$
\end{cor}
\begin{proof}
A non-split extension of $S^\lambda$ corresponds to a Hemmer element in $M^\lambda$. If $u$ and $v$ are similar Hemmer elements, then there is some $\alpha$ such that $\psi_{1,j}(u-\alpha v)=0$ for all $j$. Then $u-\alpha v\in S^\lambda$ by \Cref{Kernel Intersection Theorem}, and thus the extensions they define are the same and $\text{dim}(H^1(\mathcal{S}_{n},S^\lambda))=1$. Similarly, in the case where $\lambda$ is pointed, we may have Hemmer elements $u$ and $v$, which are not similar. Without loss of generality we may assume that $u=v+f_\lambda$ by subtracting off some $v'\in S^\lambda$, in which case the extensions $\langle S^\lambda, u \rangle $ and $\langle S^\lambda, v \rangle $ are equivalent and $\text{dim}(H^1(\mathcal{S}_{n},S^\lambda))=1$.  If $\lambda$ is neither pointed or James, then there are no Hemmer elements in $M^\lambda$, by \Cref{combinatorial main theorem}, and thus $\text{dim}(H^1(\mathcal{S}_{n},S^\lambda))=0$.
\end{proof}
Observe that the above result recovers the results of Hemmer \cite{Hemmer2}, and Donkin and Geranios for the case of two part partitions \cite{DG}, however our proof is entirely in the setting of the symmetric group. In the next section we shall go further by describing how to construct a Hemmer element, $u$ such that the extension, $\langle S^\lambda , u \rangle$ is non-split.  

\section{Non-split Extensions}
In this section we shall describe how to construct a Hemmer element $u\in M^\lambda$, when $\lambda=(a,b)$ is either pointed or James. The extension of the Specht module by this element $\langle S^\lambda , u \rangle$ is non-split. By \Cref{combinatorial main theorem} this is only possible if $(a,b)$ is James or pointed, and the element described is unique up to similarity if $(a,b)$ is James, while it is unique up to similarity and the addition of a constant design if $(a,b)$ is pointed. In either of these cases the extension is unique, up to equivalence, by \Cref{main result}. 

\subsection{$(a,b)$ pointed}

First we construct a Hemmer element for $(a,b)$ where $b$ is a $p$ power and $a\not\equiv -1  \hspace{2pt} ( \emph{mod } p^{l_p(b-1)} )$, extending a result of Nguyen \cite{Ha Thu}, who solved the case when $(a,b)=(rp^\beta,p^\beta)$ for $r\le p-1$. We remark here that Nguyen constructs a candidate for $u$ in the case that $\lambda=(a,p^\beta)$ and $a\equiv -1  \hspace{2pt} ( \emph{mod } p^{\beta+1} )$, however this element is not Hemmer as it does not satisfy the second condition of \Cref{Hemmer}. We first give an small example, in order to introduce some notation and to illustrate the more general example.

When $\lambda=(a,b)$ is a two part partition, there is a natural bijection between $\lambda$-tabloids and subsets of $[a+b]$ of size $b$. We shall make use of this throughout and identify the tabloid whose second row contains the elements $x_1,\dots,x_b$ with the set $\{x_1,\dots,x_b\}$  
Let $p=3$ and let $\lambda=(3,3)\vdash 6$. Define $u\in M^{(3,3)}$ by $u=\sum \{t\}$, where the sum is over all $\{t\}$ with 1 appearing in the top row. That is; 
\begin{align*}
 u&=\{2,3,4\}+\{2,3,5\}+\{2,4,5\}+\{3,4,5\}+\{2,3,6\}\\
 &+\{2,4,6\}+\{3,4,6\}+\{2,5,6\}+\{3,5,6\}+\{4,5,6\}.  
\end{align*} 
Observe that 
\begin{align*}
 \psi_{1,2}(u)&=3\cdot(\{2,3\}+\{2,4\}+\{3,4\}+\{2,5\}+\{3,5\}
 \\& \hspace{48pt}+\{4,5\}+\{2,6\}+\{3,6\}+\{4,6\}+\{5,6\})\\
 &=0\\
 \psi_{1,1}(u)&=6\cdot(\{2\}+\{3\}+\{4\}+\{5\}+\{6\})\\
 &=0\\
 \psi_{1,0}(u)&=10\cdot\emptyset\\
 &=-f_{(6)}.
\end{align*} 
While 
\begin{align*}
 \psi_{1,2}(f_{(3,3)})&={4\choose 1}f_{(4,2)}\\
 &=f_{4,2}\\
 \psi_{1,1}(f_{(3,3)})&={5\choose 2}f_{(5,1)}\\
 &=f_{(5,1)}\\
 \psi_{1,0}(u)&={6\choose 3}f_{(6)}\\
 &=-f_{(6)}.
\end{align*}
Clearly there is no scalar $c\in k$ such that all the $\psi_{i,v}(c\cdot f_{(3,3)} - u)=0$ and thus $u$ is Hemmer. 
\begin{remark}
This is a different example to the Hemmer element constructed in \cite[4.1]{Hemmer}, but the difference $u-v$, where $v$ is the Hemmer element from \cite[4.1]{Hemmer}, is similar to $f_\lambda$ when we consider these as designs.
\end{remark}
A similar construction can be used whenever $b$ is a $p$-power:
\begin{prop}
Let $(a,b)$ be a partition with $b=p^\beta$ and $a\not\equiv -1  \hspace{2pt} ( \emph{mod } p^{\beta-1} )$. Let $u$ be the sum of all tabloids which have the entries $1,2,\dots,a-b+1=:m$ all appearing in the top row, then $u$ is Hemmer.  
\end{prop}
\begin{proof}
The coefficient of a set $X$ of size $v$ in $\psi_{1,v}(u)$ is 0 if $X\cap [m]\ne \emptyset$ and is ${a+b-m-v \choose b-v}$ otherwise. But $m$ was chosen such that the binomial coefficients ${a-m+1 \choose 1},{a-m+2 \choose 2},\dots,{a-m+b-1 \choose b-1}$ are all divisible by $p$ (see \Cref{binomial}) and thus $\psi_{1,v}(u)=0$ for all $v\ge 1$. The coefficient ${a-m+b \choose b}$ is non-zero, and thus $\psi_{1,0}(u)\ne 0$ is a scalar multiple of $f_{(a+b)}$. If $a\equiv -1 \hspace{2pt} ( \emph{mod } p^{l_p(b-1)} )$, then $u$ will not satisfy condition 2 of \Cref{Hemmer}, otherwise $u$ is Hemmer.  
\end{proof}

In \cite{Jolliffe} we describe how to modify $u=\sum_{X\in u} X$ to create a Hemmer element for $(a,p^\beta+\hat{b})$. We shall include the construction here for completeness.
Let $Y=\{a+p^\beta+1,\dots,a+b\}$, then $Y$ is a set of size $\hat{b}$. Let $u_Y\sum_{X\in u} X\cup Y$ be the element in $M^{(a,p^\beta)}$ obtained by adjoining $Y$ to the bottom row of all tabloids appearing in $u$. Similarly $u^Y$ is obtained by adjoining $Y$ to the top row of all tabloids appearing in $u$. Consider $\psi_{1,j}(u_Y)$, which is a formal sum of sets of size $j$, by grouping terms by the size of their intersection with $Y$. First, consider the case where $\hat{b}<j<b$:
$$
\psi_{1,j}(u_Y)=\psi_{1,j-\hat{b}}(u)_Y+\sum_{y\in Y}\psi_{1,j-\hat{b}+1}(u)_{Y\backslash \{y\}}^y+\cdots+\psi_{1,j}(u)^Y.
$$
Each of these terms is $0$, by our choice of $u$, so $\psi_{1,j}(u_Y)=0$. Similarly for $j\le \hat{b}$
\begin{align*}
\psi_{1,j}(u_Y)&=\sum_{i=0}^j \sum_{\mid Y'\cap Y\mid=i}\psi_{1,j-i}(u)_{Y'}\\
&=\sum_{\mid Y'\cap Y\mid=j}\psi_{1,0}(u)_{Y'}\\
&=\mu_0 \sum_{\mid Y'\cap Y\mid=j}Y',
\end{align*}
where $\mu_0\ne 0$ is the coefficient of $u$ as a 0-design. Observe that if $Y$ is any subset of $[a+b]$, then we may define $u_Y$ similarly, by relabeling $u$ so that it has entries in $[a+b]\backslash Y$.

Let $X\subseteq [a+b]$ of size $b-1=p^{\beta}+\hat{b}-1$. Define $u_{\Bar{X}}:=\sum_{Y\subseteq X}u_Y$. Then 
$$
\psi_{1,j}(u_{\Bar{X}})=\sum_{Y\subseteq X}\psi_{1,j}(u_Y),
$$
which is $0$ if  $\hat{b}<j<b$. When $j\le \hat{b}$,
\begin{align*}
\psi_{1,j}(u_{\Bar{X}})&=\sum_{Y\subseteq X}\psi_{1,j}(u_Y)\\
                     &=\sum_{Y\subseteq X}\sum_{Y'\subseteq Y}\mu_0 Y'\\
                     &={p^\beta-1+\hat{b}-j \choose \hat{b}-j}\mu_0\sum_{Y'\subseteq X} Y',\\
\end{align*}
which is 0 if $j\ne \hat{b}$. So
$$
\psi_{1,\hat{b}}(u_{\Bar{X}})=\mu_0\sum_{Y'\subseteq X}Y',
$$
where the sum is over all subsets $Y'\subseteq X$ of size $\hat{b}$.
If $\mathbf{U}$ is a non null $p$-ary $\hat{b}$-design of block size $b-1$ and coefficient $\alpha$, which exists due to a result of Wilson \cite{Wilson}, then setting
$$
u_\mathbf{U}:=\sum_{X}\mathbf{U}
(X)u_{\Bar{X}},
$$
where the sum is over all sets $X$ of size $b-1$ and $\mathbf{U}(X)$ is the coefficient of $X$ in the $\hat{b}$-design $\mathbf{U}$, we see 
\begin{align*}
\psi_{1,\hat{b}}(u_\mathbf{U})&=\sum_{X}\mathbf{U}(X)\psi_{1,\hat{b}}u_{\Bar{X}}\\
              &=\sum_{X}\mathbf{U}(X)\mu_0\sum_{Y'\subseteq X}Y'\\
              &=\alpha \mu_0 \sum_{Y'\subseteq X}Y',
\end{align*}
and of course 
$$
\psi_{1,j}(u_\mathbf{U})=0
$$
for all other $j$. The analysis in \cite{Jolliffe} then gives: 
\begin{theorem}\label{pointed}
Let $\lambda=(a,b)$ be such that $b=p^\beta+\hat{b}$ and $\hat{b}<p^{\val(a+1)}<b$. Let $u_\mathcal{U}$ be as described above, then any non-split extension of $S^\lambda$ by the trivial module is equivalent to $\langle S^\lambda, u_\mathbf{U} \rangle .$  
\end{theorem}

\subsection{$(a,b)$ James}

\Cref{main result} tells us that the only other two part partitions for which we can construct Hemmer elements are James partitions. In this case, the design corresponding to the Hemmer element is \emph{integral}, which is to say that the design can be written as a function $u:[a+b]_b\to \mathbb{Z}$ and the coefficients $\mu_i$ are all in $\mathbb{Z}$. Reducing this integral design mod $p$ gives rise to a $p$-ary design. Integral designs have been extensively studied, and their existence was determined by Graver and Jurkat \cite{GJ}. Their construction of a universal integral design for $(v-b,b)$ uses functions $c$ which are not universal designs, but for which the induced function $\hat{c}$ is constant on all sets of size $i$ for each $i\le t$. We call such a function a $(v,\mu_1,\mu_2,\dots,\mu_t)$-design of block size $b$. The construction uses induction on this $t$, and by setting $t=b-1$ we obtain a universal design.

\begin{theorem}\cite[Section 5]{GJ}\label{Existence of integral designs}
Let $v,b,\mu_1,\mu_2,\dots,\mu_t$ be integers where $v\ge 1$ and $0\le t<b\le v$. There exists an integral $(v,\mu_1,\mu_2,\dots,\mu_t)$-design of block size $b$ if and only if $\mu_{s+1}=\frac{b-s}{v-s}\mu_s$ for $0\le s<t$.
\end{theorem}
The \textit{inclusion matrix,} $A_{i}^b(v)$, where $i\le b \le v$, is  the ${v \choose i} \times {v \choose b}$ matrix whose rows are indexed by subsets of $[v]$ of size $i$ and whose columns are indexed by subsets of $[v]$ of size $b$. The entry corresponding to position $X,Y$ is $1$ if $X\subseteq Y$ and $0$ otherwise. Gottlieb showed matrix is of full rank over characteristic 0 \cite{Gottlieb}. If $c$ is an integral $(v,\mu_1,\mu_2,\dots,\mu_t)$ design of block size $b$, then considering $c$ as a vector of length ${v\choose b}$, we see that 
$$
A_{i}^b(v)c=\mu_i \mathbf{1_i},
$$
where $\mathbf{1_i}$ is the vector of length ${v \choose i}$ consisting of 1's. It is clear that 
$$
A_{j}^i(v)A_{i}^b(v)={b-j\choose i-j} A_{j}^b(v),
$$
and thus 
$$
{v-i\choose i-j}\mu_i={b-j\choose i-j}\mu_j.
$$
This proves the necessity of the conditions in \Cref{Existence of integral designs}; to prove the sufficiency we need the following result:
\begin{theorem}\label{Null design module}\cite[Section 4]{GJ}
Let $0\le t < b \le v-t$ and denote by $N_{t,b}$ the set of all null $(v,\mu_1,\dots,\mu_t)$ designs of block size $b$. Then $A_{t+1}^b(N_{t,b})=N_{t,t+1}$.
\end{theorem}
\begin{proof}[Proof of \Cref{Existence of integral designs}]
We have already seen that the conditions are necessary. We shall prove sufficiency of the conditions by induction on $t$, noting that if $t=0$ then the design which assigns $\mu_0$ to the set $[b]$ and $0$ to all other sets of size $b$ is of the form we seek. Now assume that these conditions are sufficient for $t\ge 0$, and that $\mu_1,\mu_2,\dots,\mu_{t+1}$ satisfy these conditions. Then there is some $(v,\mu_1,\mu_2,\dots,\mu_{t})$-design, $c'$, of block size $b$. We shall construct $c$ a $(v,\mu_1,\mu_2,\dots,\mu_{t+1})$-design, $c$, of block size $b$. If $b\ge v-t$ then $A_{t,l}$ is of full column rank and thus the only designs are multiples of the constant design. In this case $c'=\alpha \mathbf{1_b}$ is also a $(v,\mu_1,\mu_2,\dots,\mu_t,\mu_{t+1}')$-design. The relationship between the coefficients of the design established previously ensure that $\mu_{t+1}=\mu_{t+1}'$.

We now consider the case where $b<v-t$. Observe,
\begin{align*}
A_t^{t+1}A_{t+1}^bc'&=(l-t)A_t^bc' \\
                    &=(l-t)\mu_t \mathbf{1_t}\\
                    &=A_t^{t+1}\frac{l-t}{v-t}\mu_t \mathbf{1_{t+1}}\\
                    &=A_t^{t+1}\mu_{t+1} \mathbf{1_{t+1}},
\end{align*} 
thus $d':=A_{t+1}^bc'-\mu_{t+1} \mathbf{1_{t+1}}\in N_{t,t+1}$. By \Cref{Null design module} there is a $d\in N_{t,l}$ such that $A_{t+1}^ld=d'$. Setting $c=c'-d$ we see that $$A_{t+1}^lc=A_{t+1}^lc'-d'=\mu_{t+1} \mathbf{1_{t+1}},$$ and the relationship between coefficients ensures this is a $(v,\mu_1,\mu_2,\dots,\mu_t,\mu_{t+1}')$-design, as required.
\end{proof}
An element $u\in M^{(a,b)}$ which is an $(a+b,\mu_1,\mu_2,\dots,\mu_{b-1})$-design of block size $b$ satisfies the first condition of \Cref{Hemmer} as long as one of the $\mu_i$ is non-zero (in $k$). We have to take care that when we construct such an element that $u$ also satisfies the second condition of \Cref{Hemmer}.
\begin{theorem}\label{existence of Hemmer element for James partition}
Let $\lambda=(a,b)$, then there exists an integral design which corresponds to a Hemmer element if and only if $\lambda$ is James.
\end{theorem}
\begin{proof}
Any integral design must have coefficients satisfying the conditions of \Cref{Existence of integral designs}, $\mu_{s+1}=\frac{b-s}{a+b-s}\mu_s$ for $0\le s<t$. This means that
$$
\mu_s=\frac{{a+b-s\choose a}}{{a+b\choose a}}\mu_0.
$$
To ensure that some $\mu_i\not\equiv 0 \modp$ we must take $\mu_s=c\frac{{a+b-s\choose a}}{p^d}$ where $c\in k$ is non-zero and $d$ is the least power of $p$ dividing some ${a+b-s\choose a}$ for $s\in \{0,1,\dots,b-1\}$. That is, $d=\text{min}_{s<b}\{\val{a+b-s\choose b}\}$. Observe that 
\begin{align*}
    \psi_{1,j}(f_{(a,b)})&={a+b-j \choose b-j}f_{(a+b-j,j)}\\
&=c^{-1}p^d\mu_jf_{(a+b-j,j)},
\end{align*}
 and so if $p^d$ is a unit in $k$, that is if $d=0$, then $$\psi_{1,j}(c\cdot f_{(a,b)}-u)=0,$$ and $u$ is not Hemmer. This means $u$ is Hemmer if and only if $p\mid {a+b-j\choose a}$ for all $j \in \{0,1,\dots,b-1\}$, which by \Cref{binomial} is if and only if $\lambda$ is James. 
\end{proof}
The analysis in \cite{Jolliffe} gives the following result:
\begin{theorem}\label{James}
Let $\lambda=(a,b)$ be James and let $u\in M^\lambda$ be a Hemmer element corresponding to an integral design with coefficients $$\mu_s=\frac{{a+b-s\choose a}}{p^d},$$ where $d=\text{min}_{s<b}\{\val{a+b-s\choose b}\}$. Note that \Cref{existence of Hemmer element for James partition} ensures such an element exists. Any non-split extension of $S^\lambda$ by the trivial module is equivalent to $\langle S^\lambda, u \rangle .$  
\end{theorem}

We summarise the results of this section as follows:
\begin{theorem}
Let $\lambda=(a,b)$ be a partition, and suppose $\langle S^\lambda,x\rangle$ is a non-split extension of the Specht module by the trivial module. Then either 
\begin{itemize}
     \item $\lambda$ is pointed and $\langle S^\lambda,x\rangle$ is equivalent to $\langle S^\lambda,u \rangle$, where $u$ is the Hemmer element $u_\mathbf{U}$ described in \Cref{pointed}, or
    \item $\lambda$ is James and $\langle S^\lambda,x\rangle$ is equivalent to $\langle S^\lambda,u \rangle$, where $u$ is the Hemmer element described in \Cref{James}.
\end{itemize}
\end{theorem}
\section{General Partitions}

The results in the previous section utilise the correspondence between Hemmer elements for two part partitions and combinatorial designs. The combinatorial objects which correspond to Hemmer elements for general partitions are much more complicated than designs, although they are built up from designs in some sense. If we take a Hemmer element for a partition $\lambda$ and restrict our attention to a pair of adjacent rows in $\lambda$ then we obtain a $p$-ary design. This allows us to use the results of the previous section to determine bounds on the degree of the cohomology. 

\begin{lemma}
 Let $\lambda=(\lambda_1,\dots,\lambda_r)$ and suppose $u$ is a Hemmer element in $M^\lambda$ with $\psi_{i,v}(u)=c\cdot f_{\lambda'}$ for some $c\ne 0$, then $(\lambda_i,\lambda_{i+1})$ is James or pointed. If $j\le i-2$ or $j\ge i+2$, then $(\lambda_j,\lambda_{j+1})$ is a James partition.
\end{lemma}
\begin{proof}
Observe that we may group the terms appearing in $u$, which are tabloids, by the union of the entries appearing in the $i$th and $(i+1)$th rows. Restricting our attention to any one of these groupings we see that we have an element $u'$ with $\psi_{i,l}(u')=c_l\dot f_{\lambda^{l}}$ and $c_v=c\ne 0$. In particular, $u'$ is a universal $p$-ary design. It can not be the constant design as then $u$ would not satisfy condition 2 of \Cref{Hemmer}. We conclude that $u'$ is Hemmer, and thus the partition $(\lambda_i,\lambda_{i+1})$ is James or pointed.

Now suppose that $j$ is such that $\mid i-j\mid$. Then $\psi_{j,l}(u)$ is some scalar multiple of $f_{\lambda'}$, but agrees with $u$ on rows $i$ and $i+1$. As $u$ is not the constant design on restriction to these two rows we must have $\psi_{j,l}(u)=0$. Denote by $\hat{\lambda}$ the partition obtained by deleting the rows $i$ and $i+1$ from $\lambda$. Observe that as $\psi_{i,v}(u)=c\cdot f_{\lambda'}$ is non zero, every possible $\hat{\lambda}$ tabloid must appear in $u$ with the same (non-zero) multiplicity. Thus the coefficient of any tabloid appearing in $\psi_{j,l}(u)$ is the product of some $c\ne 0$ and ${ \lambda_j+\lambda_{j+1}-s \choose \lambda_{j+1}-s}$. As $\psi_{j,l}(u)=0$ we conclude that ${ \lambda_j+\lambda_{j+1}-s \choose \lambda_{j+1}-s}=0$ for all $s$, and hence $(\lambda_j,\lambda_{j+1})$ is James. 
\end{proof}

\begin{cor}\label{weak bound}
Let $\lambda=(\lambda_1,\dots,\lambda_r)$. Then $\text{dim}H^1(\Sn,S^\lambda)\le k$, where $k$ is the number of rows $i$ for which $(\lambda_i,\lambda_{i+1})$ is James or pointed, and $(\lambda_j,\lambda_{j+1})$ is James for all $j\in [r-1]$ with $|i-j|\ge 2$. 
\end{cor}
\begin{proof}
If $u$ is a Hemmer element in $M^\lambda$ then $u$ restricted to each pair of rows $i$ and $i+1$ is either Hemmer, similar to the constant design, or is mapped to zero by each $\psi_{i,s}$. By \Cref{main result}, $u$ is uniquely determined, up to similarity and constants, and hence the extension $\langle S^\lambda,u\rangle $ is determined up to equivalence, by the rows $i$ and $i+1$ for which $u$ is Hemmer. If $u$ is Hemmer when restricted to rows $i$ and $i+1$ then $(\lambda_i,\lambda_{i+1})$ is James or pointed, and $(\lambda_j,\lambda_{j+1})$ is James for all $j\in [r-1]$ with $|i-j|\ge 2$. 
\end{proof}
\begin{cor}\label{Weber}
Let $\lambda=(\lambda_1,\dots,\lambda_r)$ with $r\ge 5$, and let $2<i+2<j<r$. If $(\lambda_i,\lambda_{i+1})$ and $(\lambda_j,\lambda_{j+1})$ are both not James, then $H^1(\Sn,S^\lambda)=0$. 
\end{cor}
\begin{remark}
This is the main result of \cite{Weber}, and gives a large class of partitions for which the first degree cohomology is trivial.
\end{remark}

For the remainder of the paper we shall carefully analyse Hemmer elements for general partitions, and will improve \Cref{weak bound}. We remark here that Donkin and Geranios have determined exactly the dimension of $H^1(\Sn,S^\lambda)$ for all partitions, and this is the bound we obtain from our combinatorial methods. We shall not explicitly construct Hemmer elements for general partitions, and so our result only gives an upper bound on the dimension of $H^1(\Sn,S^\lambda)$. Future work which constructs Hemmer elements would complete the entirely combinatorial proof of two of the three main theorems of \cite{DG}, namely Theorem 12.29 and Theorem 12.30. The last of the main results of \cite{DG}, Theorem 12.31, can not be obtained by the methods in this paper, as over fields of characteristic 2 non-split extensions of $S^\lambda$ do not necessarily have to be isomorphic to submodules of $M^\lambda$. 

We have already observed in \Cref{weak bound} that $\text{dim}H^1(\Sn,S^\lambda)$ is bounded above by the number of pairs of consecutive rows as, for a Hemmer element $u$ and a fixed $i$, the coefficients in $\psi_{i,l}(u)$ are related. In general, however, there may be relationships between the coefficients in $\psi_{i,l}(u)$ and $\psi_{j,m}(u)$, in which case we will say the pairs of rows $(i,i+1)$ and $(j,j+1)$ are \emph{dependent}. This dependence is clearly an equivalence relation, and then an improved bound is that $\text{dim}H^1(\Sn,S^\lambda)$ is at most the number of equivalence classes of dependent pairs of rows. We will now investigate when pairs of rows are dependent.

Let $\lambda=(a,b,c)$ and suppose that the pairs $(a,b)$ and $(b,c)$ are independent. Then there must be Hemmer elements $u,v\in M^\lambda$ with $\psi_{1,i}(u)=0$ and $\psi_{2,j}(v)=0$ for all $1\le i\le b$ and $1\le j \le c$. In particular, $$u\in \bigcap_{i=1}^b \text{ker}(\psi_{1,i})=S^{(a,b),(a,b,c)}$$ and $$v\in \bigcap_{j=1}^c \text{ker}(\psi_{2,j})=S^{(b,c),(a,b,c)}.$$
As the Hemmer elements $u$ and $v$ are such that $\psi_{1,l}(u)\ne 0$ and $\psi_{1,k}(v)\ne 0$ for some $l,k$, the partitions $(a,b)$ and $(b,c)$ are either James or pointed. The structure of these modules was studied in \cite[Chapter 17]{James}, where a filtration by Specht modules was obtained. 

$S^{(a,b),(a,b,c)}$ has a filtration by Specht modules with factors $S^{(a+x,b+y,c-x-y)}$ with $x+y\le c$ and $b+y\le a$. As $u\in S^{(a,b),(a,b,c)}$ the extension $\langle S^\lambda, u \rangle$ is a submodule of $S^{(a,b),(a,b,c)}$. As $u$ is Hemmer this is an extension of $S^{(a,b,c)}$ by the trivial module, and thus some factor of $S^{(a,b),(a,b,c)}$, not including the bottom factor $S^\lambda$, must have the trivial module as a submodule. This only occurs when one of these factors is James, by \Cref{H0}, and so $u$ can only exist if there is some $x$ and $y$  with $x+y\le c$ and $b+y\le a$ for which $(a+x,b+y,c-x-y)$ is a James partition. 

Similarly, the Specht module $S^{(b,c),(a,b,c)}$ has a filtration by Specht modules with factors $S^{(a+x+y,b-x,c-y)}$ with $y\le c \le b-x$. If there is a Hemmer element $v\in S^{(b,c),(a,b,c)}$, then some factor of $S^{(b,c),(a,b,c)}$, not including the bottom factor $S^{(a,b,c)}$, must have the trivial module as a submodule. Thus $v$ can only exist if there is some $x$ and $y$  with $y\le c \le b-x$ for which $(a+x+y,b-x,c-y)$ is a James partition. 

\begin{lemma}\label{p-power-1}
The pairs $(a,b)$ and $(b,c)$ are dependent if $(a,b,c)$ is James and $b = p^{\val (a+1)}-1$.
\end{lemma}
\begin{proof}
We shall prove that there can be no Hemmer element $u\in  S^{(a,b),(a,b,c)}$. As $(a,b,c)$ is a James partition, $\val(a+1)>l_p(b)$ and $\val(b+1)>l_p(c)$. The partition $(a+x,b+y,c-x-y)$ is James if $(a+x,b+y)$ and $(b+y,c-x-y)$ are both James. It is clear that $l_p(b+y)$ is either $l_p(b)$ or $l_p(b)+1$, and thus $(a+x,b+y)$ can only be James if $\val(x)>l_p(b)$. Of course $x\le c$ and so we must have $x=0$. The only factors of $S^{(a,b),(a,b,c)}$ indexed by James partitions are those of the form $S^{(a,b+y,c-y)}$ with $y\le \min{\{c,a-b\}}$. The partition $(b+y,c-y)$ is only James if $y\equiv c \quad (\textit{mod } p^l)$ for some $l$ (recall $\val(b+1)>l_p(c)$). Finally observe that $b+y\le a$ as long as $b\ne p^{\val (a+1)}-1$, and if $b = p^{\val (a+1)}-1$ then no such $u$ exists.
\end{proof}

\begin{remark}
We have actually shown something stronger: If $(a,b,c)$ is James and $b \ne p^{\val (a+1)}-1$ then $S^{(a,b),(a,b,c)}$ has a Specht module factor, other than $S^{(a,b,c)}$, with a trivial submodule. This tells us exactly where to look to find the Hemmer element $u$ described above.  
\end{remark}

\begin{lemma}\label{(ab) pointed}
Let $(a,b)$ be pointed and $(b,c)$ either pointed or James. Then the pairs $(a,b)$ and $(b,c)$ are dependent. 
\end{lemma}
\begin{proof}
As before, we shall prove that there can be no Hemmer element $u\in  S^{(a,b),(a,b,c)}$. Let $b=p^\beta+\hat{b}$ with $\hat{b}<p^{\val(a+1)}<p^\beta$. For the partition $(a+x,b+y,c-x-y)$ to be James, we must have $\val(a+x+1)>l_p(b+y)\ge \beta$. As the partition $(b,c)$ is either James or pointed and $x\le c$, we have that $\val(a+x+1)\le \val(a+1)<\beta$, so this partition is never James.
\end{proof}

\begin{lemma}\label{(James Pointed)}
Let $(a,b)$ be James and $(b,c)$ be pointed. Then the pairs $(a,b)$ and $(b,c)$ are dependent. 
\end{lemma}
\begin{proof}
This time we shall prove that there can be no Hemmer element $v\in  S^{(b,c),(a,b,c)}$. Let $c=p^\gamma+\hat{c}$ with $\hat{c}<p^{\val (b+1)}<p^\gamma$ and observe $\val(a+1)>l_p(b)$. There is no Hemmer element $v$ as long as no partition $(a+x+y,b-x,c-y)$ is James for  $y\le c \le b-x$. Observe that $\val(a+x+y+1)=\val(x+y)$, and so if $(a+x+y,b-x)$ is James then $y=0$. The partition $(b-x,c)$ is James if and only if $\val(b-x+1)>\gamma$, but then $(a+x,b-x)$ is not James, and so $v$ can not exist.  
\end{proof}

\begin{lemma}\label{lb=lc}
Let $(a,b,c)$ be James and suppose $l_p(b)=l_p(c)$. Then the pairs $(a,b)$ and $(b,c)$ are dependent. 
\end{lemma}
\begin{proof}
We shall actually show that if $(a,b,c)$ is James then $S^{(b,c),(a,b,c)}$ has a Specht module factor, other than $S^{(a,b,c)}$, with a trivial submodule if and only if $l_p(b)\ne l_p(c)$. As $(a,b,c)$ is James, $\val(a+1)>l_p(b)$ and $\val(b+1)>l_p(c)$.  Again, $\val(a+x+y+1)=\val(x+y)$, and so if $(a+x+y,b-x)$ is James then $y=0$. Subject to the condition that $x\ge b-c $ and that $(b-x,c)$ is James, the partition $(a+x,b-x)$ is James if and only if $x\equiv b \quad (\textit{mod } p^l)$ for some $l\ge l_p(c)$. This can be satisfied (for non-zero $x$) if and only if $l_p(b)\ne l_p(c)$.
\end{proof}
In order to state the main results of this paper, we must make some definitions to capture when pairs of rows in a partition are dependent. We shall follow \cite{DG} in making the following definitions. Let $\lambda=(\lambda_1,\dots,\lambda_n)$ be a James partition.
\begin{definition}
The \emph{segments} of $\lambda$ are the equivalence classes on $\{1,\dots, n\}$ generated by the equivalence relation $r\sim s$ if $l_p(\lambda_r)=l_p(\lambda_s)$.
\end{definition}
We shall call two integers $1\le r,s \le n$ \emph{adjacent} if they are in the same segment, or if $1<r<n$, $s=r+1$ and $s$ is the only element in its segment, and $\lambda_r=p^{\val(\lambda_{r-1}+1)}-1$.
\begin{definition}
The $p$-\emph{segments} of $\lambda$ are the equivalence classes on $\{1,\dots, n\}$ generated by adjacency.
\end{definition}

\begin{theorem}
Let $\lambda=(\lambda_1,\dots,\lambda_n)$ be a James partition of length $n\ge 2$. Let $k$ denote the number of $p$-segments of $\{1,\dots,n\}$. 
\begin{enumerate}
\item If $l_p(\lambda_1)=l_p(\lambda_2)$, then $\text{dim}(H^1(\mathcal{S}_{n},S^\lambda))\le k$. 
\item If $l_p(\lambda_1)>l_p(\lambda_2)$, then $\text{dim}(H^1(\mathcal{S}_{n},S^\lambda))\le k-1$.
\end{enumerate}
\end{theorem}
\begin{proof}
Observe that if $i$ and $i+1$ in the same segment then $(\lambda_{i-1},\lambda_{i})$ and $\lambda_{i},\lambda_{i+1})$ are dependent, by \Cref{lb=lc}. If $i$ and $i+1$ are in the same $p$-segment, but not the same segment, then $\lambda_i=p^{\val(\lambda_{i-1}+1)}-1$ and $l_p(\lambda_{i-1})>l_p(\lambda_{i})>l_p(\lambda_{i+1})$. In this situation, then by \Cref{p-power-1}, $(\lambda_{i-1},\lambda_{i})$ and $(\lambda_{i},\lambda_{i+1})$ are dependent. Thus for any $i$ and $j$ lying in the same $p$-segment $(\lambda_{i-1},\lambda_{i})$ and $(\lambda_{j-1},\lambda_{j})$ are dependent. If $l_p(\lambda_1)=l_p(\lambda_2)$, then $1$ and $2$ are in the same $p$-segment and the number of equivalence classes of dependent pairs of rows is $k$, while if $l_p(\lambda_1)>l_p(\lambda_2)$ then $1$ and $2$ are in different $p$-segments, and as there are pairs of rows corresponding to the $p$-segment $\{1\}$, the number of equivalence classes of dependent pairs of rows is $k-1$.
\end{proof}
We have already seen that if $\lambda$ is not James, then Hemmer elements can not exist unless the only pairs of non-James rows are close together (\Cref{Weber}). We shall now prove something even stronger, namely that if $\lambda$ is a non-James partition and $u\in M^\lambda$ is a Hemmer element, then $u$ is unique, up to similarity, or equivalently:

\begin{theorem}
Let $\lambda$ be a non-James partition. Then $\text{dim}(H^1(\mathcal{S}_{n},S^\lambda))\le 1$.
\end{theorem}
\begin{proof}
If $u\in M^\lambda$ is Hemmer and $(\lambda_j,\lambda_{j+1})$ is non-James, then $\psi_{i,l}(u) = 0$ for all $i$ such that $\mid i-j \mid \ge 2$. Thus $\psi_{i,l}(u)$ is only possibly non-zero for $i\in \{j-1,j,j+1\}$. It follows from \Cref{(ab) pointed} that $(\lambda_{j-1},\lambda_j)$ and $(\lambda_{j},\lambda_{j+1})$ are dependent and from \Cref{(James Pointed)} that $(\lambda_{j},\lambda_{j+1})$ and $(\lambda_{j+1},\lambda_{j+2})$ are dependent. This implies that Hemmer elements can not be chosen independently, as far from the pointed pair we have $\psi_{i,l}(u)=0$, while the value of $\psi_{i,l}(u)$ close to the pointed pair is determined by the coefficient of $\psi_{j,l}(u)=c\cdot f_{\lambda'}$ for some $l$ such that $c\ne 0$. 
\end{proof}

In his paper \cite{Weber}, Weber remarks that ``the strength of Hemmer's method does not lie in proving non-trivial but trivial first cohomology". After obtaining the results of this paper via the method of Hemmer we may take this remark further and say that the strength of Hemmer's method is actually in determining an upper bound to the dimension of the first cohomology (and in particular determining when it is trivial). Weber laments that the construction of Hemmer elements is difficult in practice, and although we have constructed Hemmer elements for two part partitions, where they exist, and given clues as to where to find them in general, there is still work to be done to complete the work envisioned by Hemmer in \cite{Hemmer}.

\section*{Acknowledgements}
This work will appear in the author's PhD thesis prepared at the University of Cambridge and supported by the Woolf Fisher Trust and the Cambridge Trust. This work was done while the author was a visiting scholar at Victoria University of Wellington. The author would like to thank his supervisor Dr Stuart Martin for his encouragement and support.

\thispagestyle{footer}
\end{document}